\documentclass[sn-mathphys,Numbered]{sn-jnl}

\usepackage{graphicx}%
\usepackage{multirow}%
\usepackage{amsmath,amssymb,amsfonts}%
\usepackage{amsthm}%
\usepackage{mathrsfs}%
\usepackage[title]{appendix}%
\usepackage{xcolor}%
\usepackage{textcomp}%
\usepackage{manyfoot}%
\usepackage{booktabs}%
\usepackage{algorithm}%
\usepackage{algorithmicx}%
\usepackage{algpseudocode}%
\usepackage{listings}%

\usepackage{nicefrac}
\usepackage{dsfont}
\usepackage{amsmath}
\usepackage{amsfonts}
\usepackage{amssymb}
\usepackage{amsthm}
\usepackage{MnSymbol}
\usepackage{mathrsfs}
\usepackage{graphicx}
\usepackage[all,2cell]{xy}
\usepackage{comment}
\usepackage{cjhebrew}

\usepackage[commandnameprefix=always]{changes}

\definechangesauthor[name=Matteo, color=red]{MP}
\definechangesauthor[name=Sara, color=teal]{SA}
\definechangesauthor[name=Pedro, color=violet]{PH}
\definechangesauthor[name=Vivien, color=orange]{VL}

\newcommand{\A}{\mathscr{A}}
\newcommand{\B}{\mathscr{B}}
\newcommand{\T}{\mathscr{T}}
\renewcommand{\S}{\mathscr{S}}
\newcommand{\E}{\mathscr{E}}
\newcommand{\X}{\mathscr{X}}

\newcommand{\I}{\mathscr{I}}
\renewcommand{\j}{\jmath}

\newcommand{\s}{\mathtt{s}}

\newcommand{\e}{\mathtt{e}}
\newcommand{\x}{\mathtt{x}}
\newcommand{\y}{\mathtt{y}}
\newcommand{\z}{\mathtt{z}}
\newcommand{\bz}{\bs{z}}
\renewcommand{\r}{\bs{r}}
\renewcommand{\path}{\mathrel{\leadsto}}

\newcommand{\ba}{\begin{array}}
\newcommand{\ea}{\end{array}}

\newcommand{\bs}{\boldsymbol}
 
\newcommand{\del}{\rotatebox[origin=c]{180}{$\setminus$}}
\newcommand{\con}{\reflectbox{$\setminus\!\!\!\setminus$}}
\newcommand*{\Scale}[2][4]{\scalebox{#1}{$#2$}}

\theoremstyle{thmstyleone}%
\newtheorem{theorem}{Theorem}
\theoremstyle{thmstyletwo}%

\theoremstyle{thmstylethree}%
\newtheorem{definition}{Definition}%

\begin{document}

\title[Coplanarity of rooted spanning-tree vectors]{Coplanarity of rooted spanning-tree vectors}


\author*[1]{\fnm{Matteo} \sur{Polettini}}\email{matteoeo@gmail.com}

\author[2]{\fnm{Pedro E.} \sur{Harunari}}

\author[3]{\fnm{Sara} \sur{Dal Cengio}}

\author[3]{\fnm{Vivien} \sur{Lecomte}}

\affil[1] {Via Gaspare Nadi 4, 40139 Bologna, Italy}

\affil[2]{Department of Physics and Materials Science, University of Luxembourg, Campus Limpertsberg, 162a avenue de la Fa\"iencerie, L-1511 Luxembourg (G. D. Luxembourg)} 

\affil[3]{Université Grenoble Alpes, CNRS, LIPhy, FR-38000 Grenoble, France}


\abstract{Employing a recent technology of tree surgery we prove a ``deletion-constriction'' formula for products of rooted spanning trees on weighted directed graphs that generalizes deletion-contraction on undirected graphs. The formula implies that, letting $\tau_\x^\varnothing$, $\tau_\x^+$, and $\tau_\x^-$ be the rooted spanning tree polynomials obtained respectively by removing an edge in both directions or by forcing the tree to pass through either direction of that edge, the vectors $(\tau_\x^\varnothing, \tau_\x^+, \tau_\x^-)$ are coplanar for all roots $\x$. We deploy the result to give an alternative derivation of a recently found mutual linearity of stationary currents of Markov chains. We generalize deletion-constriction and current linearity among two edges, and conjecture that similar results may hold for arbitrary subsets of edges.}

\keywords{Weighted directed graphs, rooted spanning trees, deletion-contraction, Markov chains}



\maketitle

\newpage


\section{Introduction}

Spanning trees and the properties of their associated polynomials are ubiquitous in physics, from electrical circuit theory~\cite{kirby2016kirchhoff,freitas2020stochastic} to equilibrium statistical mechanics~\cite{sokal2005multivariate,de2013exactly} and quantum field theory~\cite{caracciolo2004fermionic,gurau2009tree,aluffi2011feynman}. In the context of the statistical mechanics of irreversible systems based on Markov processes, undirected unweighted spanning trees can be used to define basis of observables that characterize stationary and transient behavior~\cite{polettini2015cycle}; spanning trees of (unweighted) digraphs allow one to derive the statistical large deviation rate function of the currents~\cite{carugno2022graph,polettini2015best}; and more importantly for what follows, spanning trees of weighted directed graphs provide the stationary probability distribution~\cite{hill1966studies,schnakenberg1976network,caplan1982tl,king1956schematic} (see also~\cite{khodabandehlou2022trees,lecturenotes} for recent reviews and~\cite{martins2023topologically, khodabandehlou2022exact, garilli2023fluctuation, liang2023universal, floydLearningControlNonequilibrium2024a, chunTradeoffsNumberFluctuations2023, maesHeatBoundsBlowtorch2013} for recent applications). In this manuscript we will be concerned with this latter class of graphs, with the edges connecting vertices carrying different weights in either direction, which may represent transition rates of a Markov process among vertices.

For undirected graphs (same weight in both directions), the deletion-contraction formula~\cite{brylawski1992tutte,bollobas2000contraction} allows one to express the spanning-tree polynomial of a graph with those of smaller graphs obtained by removing edges or by shrinking them down to a vertex. An analogy can be made to electrical resistor networks where a resistance is either put to $0$ or to $+\infty$~\cite{cinkir2011deletion}. For directed graphs, spanning tree polynomials need to be rooted to account for the directionality of the weights, and such ``topological'' contraction becomes contrived. In the electronic network analogy, this case is similar to the inclusion of diodes, which conduct electricity preferentially in one direction~\cite{freitas2021stochastic}, and it does not make sense to short-circuit the extremal vertices.

Our first, smaller, contribution is to prove that properly constrained rooted spanning-tree polynomials for weighted directed graphs (as per Definition~\ref{th:rsp}) satisfy a generalized ``logical'' acceptation of deletion-contraction compatible with (but less powerful than) the above acceptation (Theorem~\ref{th:gendelcon}). We call this formula deletion-constriction.

Our second and most important contribution is a ``second order'' deletion-constriction formula involving products of rooted spanning tree polynomials (Theorem~\ref{th:second}). The formula is graphically represented in Fig.\,\ref{fig:treesurgery} and explained in its caption. To prove it, we employ the technology of tree surgery introduced in Ref.\,\cite[App. D]{owen2020universal} in an analysis of nonequilibrium response for Markov processes.

A suggestive corollary of the formula is that
certain 3-vectors representing deletion-constriction through either directions of one edge span a $2$-dimensional space independently of the topology of the graph and the number of vertices (Theorem~\ref{th:coplanarity}, see Fig.\,\ref{fig:coplanarity}). Before proving the main results this latter fact is shown in the next section by a simple example.

As a third contribution we use this formulation to rederive a recent result by the Authors~\cite{harunari2024mutual} about mutual linearity of stationary mean currents of continuous-time Markov chains.
We then expand on these ideas, showing that certain $9$-vectors representing the deletion-constriction of $2$ reversible edges span a $3$-dimensional space, and use this fact to suggest the validity of mutual linearity among three Markov currents, under technical conditions that remain to be lifted.

Finally, we conjecture that the $3^n$-vectors representing the deletion-constriction through $n$ reversible edges span a $n+1$ dimensional space.

Below, the scalar product $\cdot$ is the standard Euclidean.

\subsection{Example of coplanarity}
\label{sec:example}

Let us first give a very simple analytical example of coplanarity. Consider the graph
\begin{align}
\begin{array}{c}\xymatrix{ \mathtt{a}  \ar@{-}[r]^{\star} & \mathtt{b} \ar@{-}[d] \\ 
\mathtt{d} \ar@{-}[ur] \ar@{-}[u] & \mathtt{c} \ar@{-}[l]} \end{array} \label{eq:graphgraph}
\end{align}
where $\star$ is the pinned edge that will be subjected to deletion-constriction. The rooted spanning-tree polynomials are, in words (see below for mathematical definitions), the sum over all sub-graphs where exactly one path connects any vertex to a fixed root vertex, of the product of the rates along the directed edges of the sub-graph. In the present example, they are given by the diagrammatic representations:
\begin{subequations}
\begin{align}
\tau_\mathtt{a}
& = \Scale[0.6]{
  \ba{c}\xymatrix{ \ar@{<-}[r] &  \ar@{<-}[d] \\  \ar[u] & } \ea
+ \ba{c}\xymatrix{ \ar@{<-}[r] & \ar@{<-}[d] \\ & \ar@{<-}[l]} \ea
+ \ba{c}\xymatrix{ \ar@{<-}[r] &  \\ \ar[u] & \ar[l]} \ea
+ \ba{c}\xymatrix{ \ar@{<-}[r]  & \\ \ar[ur]  & \ar[l]} \ea
+ \ba{c}\xymatrix{ \ar@{<-}[r]  & \ar@{<-}[d] \\ \ar[ur] & } \ea 
+ \ba{c}\xymatrix{ & \ar[d] \\ \ar[u] & \ar[l]} \ea 
+ \ba{c}\xymatrix{ & \ar@{<-}[d] \\ \ar@{<-}[ur] \ar[u] & } \ea
+ \ba{c}\xymatrix{   & \\ \ar@{<-}[ur] \ar[u] & \ar[l]} \ea
} \\
\tau_\mathtt{b}
& = \Scale[0.6]{\ba{c}\xymatrix{ \ar@{->}[r] &  \ar@{<-}[d] \\  \ar[u] & } \ea
+ \ba{c}\xymatrix{ \ar@{->}[r] & \ar@{<-}[d] \\ & \ar@{<-}[l]} \ea
+ \ba{c}\xymatrix{ \ar@{->}[r] &  \\ \ar[u] & \ar[l]} \ea 
+ \ba{c}\xymatrix{ \ar@{->}[r]  & \\ \ar[ur]  & \ar[l]} \ea
+ \ba{c}\xymatrix{ \ar@{->}[r]  & \ar@{<-}[d] \\ \ar[ur] & } \ea
+ \ba{c}\xymatrix{ & \ar@{<-}[d] \\ \ar@{<-}[u] & \ar@{<-}[l]} \ea
+ \ba{c}\xymatrix{ & \ar@{<-}[d] \\ \ar@{->}[ur] \ar@{<-}[u] & } \ea
+ \ba{c}\xymatrix{   & \\ \ar@{->}[ur] \ar@{<-}[u] & \ar[l]} \ea
} \\
\tau_\mathtt{c}
& = \Scale[0.6]{\ba{c}\xymatrix{ \ar@{->}[r] &  \ar@{->}[d] \\  \ar[u] & } \ea
+ \ba{c}\xymatrix{ \ar@{->}[r] & \ar@{->}[d] \\ & \ar@{<-}[l]} \ea
+ \ba{c}\xymatrix{ \ar@{<-}[r] &  \\ \ar@{<-}[u] & \ar@{<-}[l]} \ea 
+ \ba{c}\xymatrix{ \ar@{->}[r]  & \\ \ar@{<-}[ur]  & \ar@{<-}[l]} \ea
+ \ba{c}\xymatrix{ \ar@{->}[r]  & \ar@{->}[d] \\ \ar[ur] & } \ea
+ \ba{c}\xymatrix{ & \ar@{->}[d] \\ \ar@{<-}[u] & \ar@{<-}[l]} \ea
+ \ba{c}\xymatrix{ & \ar@{->}[d] \\ \ar@{->}[ur] \ar@{<-}[u] & } \ea
+ \ba{c}\xymatrix{   & \\ \ar@{<-}[ur] \ar@{<-}[u] & \ar@{<-}[l]} \ea
} \\
\tau_\mathtt{d}
& = \Scale[0.6]{\ba{c}\xymatrix{ \ar@{<-}[r] &  \ar@{<-}[d] \\  \ar@{<-}[u] & } \ea
+ \ba{c}\xymatrix{ \ar@{->}[r] & \ar@{->}[d] \\ & \ar@{->}[l]} \ea
+ \ba{c}\xymatrix{ \ar@{<-}[r] &  \\ \ar@{<-}[u] & \ar[l]} \ea 
+ \ba{c}\xymatrix{ \ar@{->}[r]  & \\ \ar@{<-}[ur]  & \ar[l]} \ea
+ \ba{c}\xymatrix{ \ar@{->}[r]  & \ar@{<-}[d] \\ \ar@{<-}[ur] & } \ea
+ \ba{c}\xymatrix{ & \ar@{->}[d] \\ \ar@{<-}[u] & \ar@{->}[l]} \ea
+ \ba{c}\xymatrix{ & \ar@{<-}[d] \\ \ar@{<-}[ur] \ar@{<-}[u] & } \ea
+ \ba{c}\xymatrix{   & \\ \ar@{<-}[ur] \ar@{<-}[u] & \ar[l]} \ea
}
\end{align}
\end{subequations}

For each rooted spanning-tree polynomial, we collect the terms according to the conditions: those that do not contain the pinned edge; those that contain $\mathtt{a}\gets \mathtt{b}$; those that contain $ \mathtt{a} \to \mathtt{b}$. We collect the sub-polynomials so obtained in $3$-vectors. Using the diagrammatic representation, we obtain:
\begin{subequations}
\begin{align}
\boldsymbol{\tau}_\mathtt{a}
& = \left(\ba{l} 
\Scale[0.6]{
\ba{c}\xymatrix{ & \ar[d] \\ \ar[u] & \ar[l]} \ea 
+ \ba{c}\xymatrix{ & \ar@{<-}[d] \\ \ar@{<-}[ur] \ar[u] & } \ea
+ \ba{c}\xymatrix{   & \\ \ar@{<-}[ur] \ar[u] & \ar[l]} \ea
} \\
\Scale[0.6]{
  \ba{c}\xymatrix{  \ar@{<-}[r] &  \ar@{<-}[d] \\  \ar[u] & } \ea
+ \ba{c}\xymatrix{  \ar@{<-}[r] & \ar@{<-}[d] \\ & \ar@{<-}[l]} \ea
+ \ba{c}\xymatrix{  \ar@{<-}[r] &  \\ \ar[u] & \ar[l]} \ea
+ \ba{c}\xymatrix{  \ar@{<-}[r] & \\ \ar[ur]  & \ar[l]} \ea
+ \ba{c}\xymatrix{  \ar@{<-}[r] & \ar@{<-}[d] \\ \ar[ur] & } \ea} \\
0 \ea \right)
\\
\boldsymbol{\tau}_\mathtt{b}
& = \left(\ba{l} 
\Scale[0.6]{
\ba{c}\xymatrix{ & \ar@{<-}[d] \\ \ar@{<-}[u] & \ar@{<-}[l]} \ea
+ \ba{c}\xymatrix{ & \ar@{<-}[d] \\ \ar@{->}[ur] \ar@{<-}[u] & } \ea
+ \ba{c}\xymatrix{   & \\ \ar@{->}[ur] \ar@{<-}[u] & \ar[l]} \ea
}
\\
0
\\
\Scale[0.6]{
\ba{c}\xymatrix{ \ar@{->}[r]  &  \ar@{<-}[d] \\  \ar[u] & } \ea
+ \ba{c}\xymatrix{ \ar@{->}[r] & \ar@{<-}[d] \\ & \ar@{<-}[l]} \ea
+ \ba{c}\xymatrix{ \ar@{->}[r] &  \\ \ar[u] & \ar[l]} \ea 
+ \ba{c}\xymatrix{ \ar@{->}[r] & \\ \ar[ur]  & \ar[l]} \ea
+ \ba{c}\xymatrix{ \ar@{->}[r] & \ar@{<-}[d] \\ \ar[ur] & } \ea
} 
 \ea \right)
\\
\boldsymbol{\tau}_\mathtt{c}
& = \left(\ba{l} 
\Scale[0.6]{
\ba{c}\xymatrix{ & \ar@{->}[d] \\ \ar@{<-}[u] & \ar@{<-}[l]} \ea
+ \ba{c}\xymatrix{ & \ar@{->}[d] \\ \ar@{->}[ur] \ar@{<-}[u] & } \ea
+ \ba{c}\xymatrix{   & \\ \ar@{<-}[ur] \ar@{<-}[u] & \ar@{<-}[l]} \ea
}
\\
\Scale[0.6]{
\ba{c}\xymatrix{\ar@{<-}[r]  &  \\ \ar@{<-}[u] & \ar@{<-}[l]} \ea 
}
\\
\Scale[0.6]{\ba{c}\xymatrix{\ar@{->}[r] &  \ar@{->}[d] \\  \ar[u] & } \ea
+ \ba{c}\xymatrix{\ar@{->}[r] & \ar@{->}[d] \\ & \ar@{<-}[l]} \ea
+ \ba{c}\xymatrix{\ar@{->}[r]  & \\ \ar@{<-}[ur]  & \ar@{<-}[l]} \ea
+ \ba{c}\xymatrix{\ar@{->}[r]  & \ar@{->}[d] \\ \ar[ur] & } \ea
}
 \ea \right)
 \\
 \boldsymbol{\tau}_\mathtt{d}
& = \left(\ba{l} 
\Scale[0.6]{
\ba{c}\xymatrix{ & \ar@{->}[d] \\ \ar@{<-}[u] & \ar@{->}[l]} \ea
+ \ba{c}\xymatrix{ & \ar@{<-}[d] \\ \ar@{<-}[ur] \ar@{<-}[u] & } \ea
+ \ba{c}\xymatrix{   & \\ \ar@{<-}[ur] \ar@{<-}[u] & \ar[l]} \ea
}
\\
\Scale[0.6]{
\ba{c}\xymatrix{\ar@{<-}[r]  &  \ar@{<-}[d] \\  \ar@{<-}[u] & } \ea
+ \ba{c}\xymatrix{\ar@{<-}[r]  &  \\ \ar@{<-}[u] & \ar[l]} \ea 
}
\\
\Scale[0.6]{
\ba{c}\xymatrix{\ar@{->}[r] & \ar@{->}[d] \\ & \ar@{->}[l]} \ea
+ \ba{c}\xymatrix{\ar@{->}[r] & \\ \ar@{<-}[ur]  & \ar[l]} \ea
+ \ba{c}\xymatrix{\ar@{->}[r] & \ar@{<-}[d] \\ \ar@{<-}[ur] & } \ea
}
 \ea \right).
\end{align}
\end{subequations}
These vectors appear to be quite arbitrary. Obviously, given that there are four vectors spanning a three-dimensional space, at least one of them is not linearly independent of the others. In fact, two of them are not: all such vectors are coplanar. This can be shown by defining the vector
\begin{align}
\bs{\sigma}_{\mathtt{a} \leftrightarrow \mathtt{b}} := \left(\ba{l}
\Scale[0.6]{
-  \ba{c}\xymatrix{ \ar@/^.3pc/@{->}[r] \ar@/_.3pc/@{<-}[r] &  \ar@{<-}[d] \\  \ar[u] & } \ea
- \ba{c}\xymatrix{ \ar@/^.3pc/@{->}[r] \ar@/_.3pc/@{<-}[r] & \ar@{<-}[d] \\ & \ar@{<-}[l]} \ea
- \ba{c}\xymatrix{ \ar@/^.3pc/@{->}[r] \ar@/_.3pc/@{<-}[r] &  \\ \ar[u] & \ar[l]} \ea
- \ba{c}\xymatrix{ \ar@/^.3pc/@{->}[r] \ar@/_.3pc/@{<-}[r] & \\ \ar[ur]  & \ar[l]} \ea
- \ba{c}\xymatrix{ \ar@/^.3pc/@{->}[r] \ar@/_.3pc/@{<-}[r] & \ar@{<-}[d] \\ \ar[ur] & } \ea 
} 
\\
\Scale[0.6]{\ba{c}\xymatrix{\ar@{->}[r] & \ar[d] \\ \ar[u] & \ar[l]} \ea 
+ \ba{c}\xymatrix{\ar@{->}[r] & \ar@{<-}[d] \\ \ar@{<-}[ur] \ar[u] & } \ea
+ \ba{c}\xymatrix{\ar@{->}[r]   & \\ \ar@{<-}[ur] \ar[u] & \ar[l]} \ea
}
\\ 
\Scale[0.6]{\ba{c}\xymatrix{\ar@{<-}[r] & \ar@{<-}[d] \\ \ar@{<-}[u] & \ar@{<-}[l]} \ea
+ \ba{c}\xymatrix{\ar@{<-}[r] & \ar@{<-}[d] \\ \ar@{->}[ur] \ar@{<-}[u] & } \ea
+ \ba{c}\xymatrix{\ar@{<-}[r]   & \\ \ar@{->}[ur] \ar@{<-}[u] & \ar[l]} \ea
}
\ea
\right)
\end{align}
and showing by a tedious hand calculation that $\bs{\sigma}_{\mathtt{a} \leftrightarrow \mathtt{b}} \cdot \bs{\tau}_\x = 0$, $\forall \x \in \{\mathtt{a}, \mathtt{b}, \mathtt{c}, \mathtt{d}\}$. 

A useful variant of this result is obtained when we further constrain spanning trees to pass through or not pass through a given edge. Let us for example consider the subset of spanning trees that pass through the oriented edge $\mathtt{c} \to \mathtt{b}$. We now obtain the new vectors
\begin{subequations}
\begin{align}
\boldsymbol{\tau}^{\con \mathtt{c} \to \mathtt{b}}_\mathtt{a}
& = \left(\ba{l} 
\Scale[0.6]{
\ba{c}\xymatrix{ & \ar@{<-}[d] \\ \ar@{<-}[ur] \ar[u] & } \ea
} \\
\Scale[0.6]{
  \ba{c}\xymatrix{ \ar@{<-}[r] &  \ar@{<-}[d] \\  \ar[u] & } \ea
+ \ba{c}\xymatrix{ \ar@{<-}[r] & \ar@{<-}[d] \\ & \ar@{<-}[l]} \ea
+ \ba{c}\xymatrix{ \ar@{<-}[r] & \ar@{<-}[d] \\ \ar[ur] & } \ea} \\
0 \ea \right)
\\
\boldsymbol{\tau}_\mathtt{b}^{\con \mathtt{c} \to \mathtt{b}}
& = \left(\ba{l} 
\Scale[0.6]{
\ba{c}\xymatrix{ & \ar@{<-}[d] \\ \ar@{<-}[u] & \ar@{<-}[l]} \ea
+ \ba{c}\xymatrix{ & \ar@{<-}[d] \\ \ar@{->}[ur] \ar@{<-}[u] & } \ea
}
\\
0
\\
\Scale[0.6]{
\ba{c}\xymatrix{ \ar@{->}[r] &  \ar@{<-}[d] \\  \ar[u] & } \ea
+ \ba{c}\xymatrix{ \ar@{->}[r] & \ar@{<-}[d] \\ & \ar@{<-}[l]} \ea
+ \ba{c}\xymatrix{ \ar@{->}[r] & \ar@{<-}[d] \\ \ar[ur] & } \ea
} 
 \ea \right)
 \\
 \bs{\tau}^{\con \mathtt{c} \to \mathtt{b}}_\mathtt{c} & = \left(\ba{c} 0 \\ 0 \\ 0 \ea \right)
\\ 
 \boldsymbol{\tau}_\mathtt{d}^{\con \mathtt{c} \to \mathtt{b}}
& = \left(\ba{l} 
\Scale[0.6]{
\ba{c}\xymatrix{ & \ar@{<-}[d] \\ \ar@{<-}[ur] \ar@{<-}[u] & } \ea
}
\\
\Scale[0.6]{
\ba{c}\xymatrix{ \ar@{<-}[r] &  \ar@{<-}[d] \\  \ar@{<-}[u] & } \ea
}
\\
\Scale[0.6]{
\ba{c}\xymatrix{\ar@{->}[r] & \ar@{<-}[d] \\ \ar@{<-}[ur] & } \ea
}
 \ea \right).
\end{align}
\end{subequations}
Similarly, we find a vector that is orthogonal to all three of them starting from $\bs{\sigma}_{\mathtt{a} \leftrightarrow \mathtt{b}}$ and removing all entries that do not pass through $\mathtt{c} \to \mathtt{b}$, yielding:
\begin{align}
\bs{\sigma}_{\mathtt{a} \leftrightarrow \mathtt{b}}^{\con \mathtt{c} \to \mathtt{b}} = \left(\ba{l}
\Scale[0.6]{
-  \ba{c}\xymatrix{ \ar@/^.3pc/@{->}[r] \ar@/_.3pc/@{<-}[r] &  \ar@{<-}[d] \\  \ar[u] & } \ea
- \ba{c}\xymatrix{ \ar@/^.3pc/@{->}[r] \ar@/_.3pc/@{<-}[r] & \ar@{<-}[d] \\ & \ar@{<-}[l]} \ea
- \ba{c}\xymatrix{ \ar@/^.3pc/@{->}[r] \ar@/_.3pc/@{<-}[r] & \ar@{<-}[d] \\ \ar[ur] & } \ea 
} 
\\
\Scale[0.6]{
\ba{c}\xymatrix{  \ar@{->}[r]  & \ar@{<-}[d] \\ \ar@{<-}[ur] \ar[u] & } \ea
}
\\ 
\Scale[0.6]{\ba{c}\xymatrix{ \ar@{<-}[r] & \ar@{<-}[d] \\ \ar@{<-}[u] & \ar@{<-}[l]} \ea
+ \ba{c}\xymatrix{ \ar@{<-}[r] & \ar@{<-}[d] \\ \ar@{->}[ur] \ar@{<-}[u] & } \ea
}
\ea
\right) .
\end{align}
Once again it can be checked by hand calculation that this latter vector is orthogonal to the ones above.

\begin{figure}
\begin{center}
  \includegraphics[width=.8\linewidth]{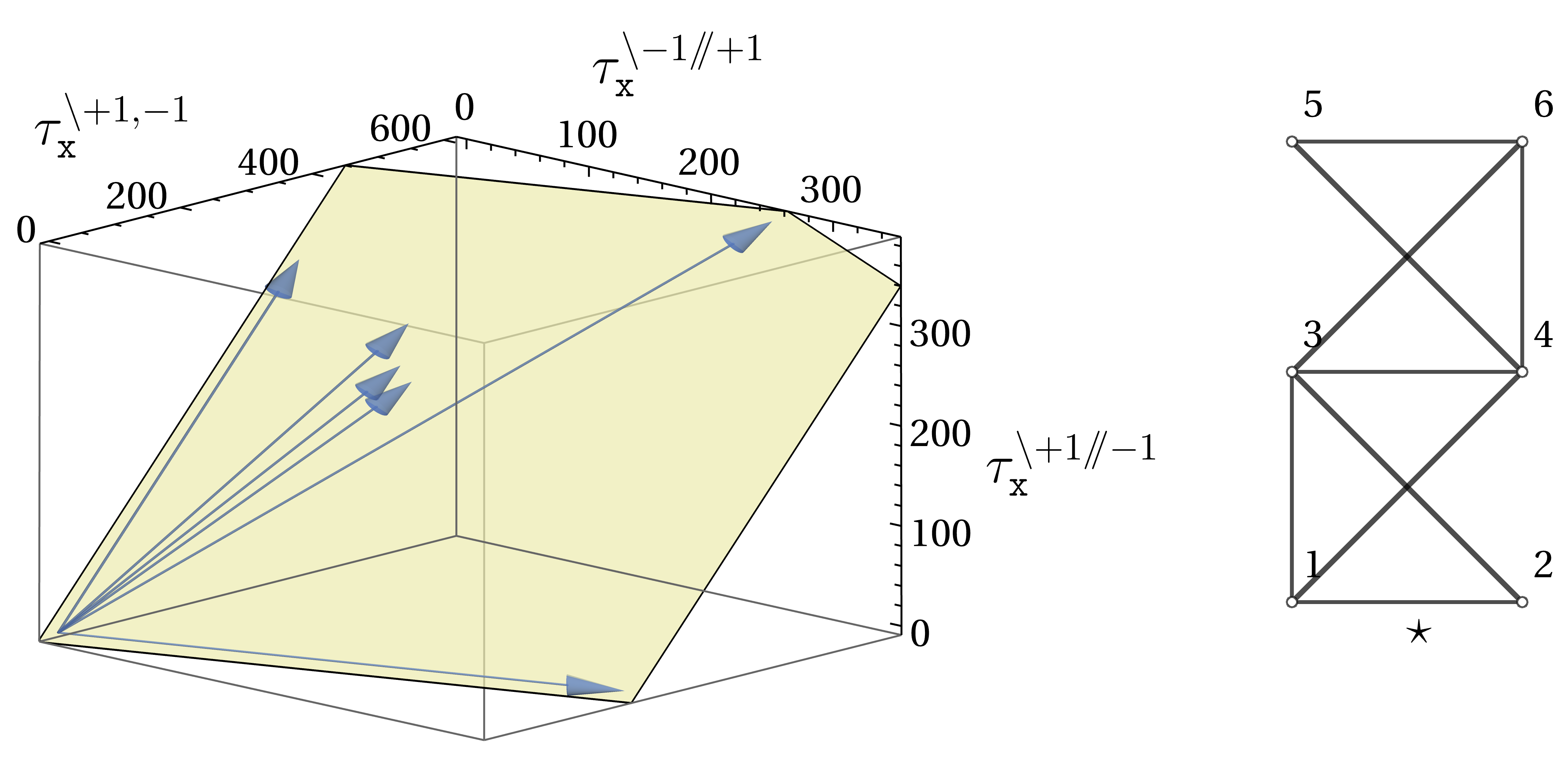}
 \end{center}
  \caption{Coplanarity of spanning-tree vectors defined in Eq.\:\eqref{eq:sigma}, for the network depicted on the right. The pinned edge is $1\mbox{---}2$ and the depicted vectors $\boldsymbol{\tau}_\x$ are rooted at each vertex $\x$ of the network, all of them lying on the plane in yellow.
  The rates were drawn randomly between $0.5$ and $3$.}
  \label{fig:coplanarity}
\end{figure}

Fig.\,\ref{fig:coplanarity} suggests that coplanarity holds also for more complicated networks, as we are going to prove.

\section{Setting, methods and results}

\subsection{Rooted spanning trees and deletion-constriction}

We consider a weighted directed graph $\mathscr{G} = (\mathscr{X},\E,r)$, where $\x \in \mathscr{X}$ are the vertices,  where the edges $\x \to \y \in \E \subseteq \mathscr{X} \times \mathscr{X}$ are ordered couples of vertices, and where $r:\E \to \mathbb{R}^+$ is a positive weight. For each couple of edges $\x \to \y$ and $\y \to \x$ we choose a reference forward orientation and denote $\pm e$ oriented edges connecting a source vertex $\s(\pm e) \in \mathscr{X}$ to a target vertex $\s(\mp e)$. In general, different weights are associated to edges $+e$ and $-e$. If $r(+e) = r(-e)$ for all edges then the graph is undirected. By extension $r$ will also be the function that takes the product of rates in a subset of edges, on the assumption $r(\varnothing) = 1$.

We assume the graph to be irreducible, that is, there exists a path $\x \path \y$ of nonzero weight between any two vertices. 

In an undirected graph a tree $\T$ is a subset of edges that contains no cycle; it is spanning if it connects all vertices. In a directed graph, a rooted spanning tree $\T_\x$ with root $\x$ is a spanning tree such that each edge is directed along the unique path that leads to the root. 

We now define a generalization of the rooted spanning tree polynomial. 

\begin{definition}
\label{th:rsp}
Let $\A$ and $\B$ be two non-intersecting subsets of $\E$. \emph{The conditioned rooted spanning tree polynomial} is defined as
\begin{align}
\tau^{\del \A \con \B}_\x = \sum_{\substack{\T_\x \\ \A \cap \T_\x = \varnothing \\ \B \subseteq \T_\x}}  r(\T_\x). \label{eq:trees}
\end{align}
\end{definition}
\noindent
Clearly it is a function of all rates. In particular, the sum spans over trees that do not contain edges in $\A$ and that must contain edges in $\B$.
Notice that the above polynomial is non-negative. Furthermore, $\tau^{\del \A \con \B}_\x$ does not depend on any of the rates in $\A$, and is multilinear in all the rates in $\B$. Consequently, we can factorize the rates in $\B$, and rewrite  
\begin{align}
\tau^{\del \A \con \B}_\x  = \dot{\tau}^{\del \A \con \B}_\x \prod_{e \in \B} r(e) \label{eq:dotted}
\end{align}
where now the rescaled spanning-tree polynomial $\dot{\tau}^{\del \A \con \B}_\x$, defined by this relation, does not depend on any of the rates in $\A$ and $\B$.

The (unconditioned) rooted spanning tree polynomial $\tau_\x = \sum_{\T_\x}  r(\T_\x)$ is obtained for $\A = \B = \varnothing$. Notice that in the case the graph is undirected $\tau^{\del \A \con \B}_\x = \tau^{\del \A \con \B}$ does not depend on the root. Furthermore, notice that one can obtain $\tau^{\del \A \con \B}_\x$ from $\tau_\x$ by taking derivatives with respect to $\log r(e)$ for all $e \in \B$, and by evaluating at $r(e') \equiv 0$ for all $e' \in \A$.

The following result holds for such polynomials.

\begin{theorem}
\label{th:gendelcon}
Let $\I \subseteq \E$ be any subset of the edge set. Then the unconditioned rooted spanning tree polynomial can be given in terms of conditioned rooted spanning-tree polynomials as 
\begin{align}
\tau_\x = \sum_{\A \subseteq \I} \tau_\x^{\del \A \con (\I \del \A)} \label{eq:gendelcon}
\end{align}
where $\A$ ranges among any subset of $\I$ (including $\emptyset$ and $\I$ itself).
\end{theorem}

\begin{proof}
Using the definition Eq.\,(\ref{eq:trees}) with $\B = \I \del \A$ and swapping sums we obtain
\begin{align}
 \sum_{\A \subseteq \I} \tau_\x^{\del \A \con (\I \del \A)} = \sum_{\T_\x} r(\T_\x) \sum_{\substack{ \A \subseteq \I \\ \A \cap \T_\x = \varnothing \\ \I \del \A \subseteq \T_\x} } 1
\end{align}
The formula follows from the fact that, for given $\T_\x$ and $\I$, there is a unique subset $\A$ that meets the criteria of the right-hand sum  $\A = \I \del (\I \cap \T_\x)$ (see diagram in Fig.\,\ref{fig:venn}), in such way that $\sum_{\substack{ \A \subseteq \I \\ \A \cap \T_\x = \varnothing \\ \I \del \A \subseteq \T_\x} } 1 = 1$.
\end{proof}

\begin{figure}
\begin{center}
  \includegraphics[width=.4\linewidth]{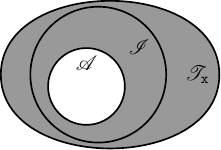}
\end{center}
  \caption{The grey area is the set of edges of the rooted spanning tree $\T_\x$, which by deletion does not contain the edges in $\A$ (smaller circle), which is a subset of $\I$; due to constriction $\I \del \A$ must belong to $\T_\x$. It follows that, given $\I$ and $\T_\x$, $\A$ is uniquely determined (possibly empty).}
  \label{fig:venn}
\end{figure}

The notation $\del$ and $\con$ is borrowed from the deletion-contraction paradigm $(\del,~\reflectbox{$\setminus$})$ for undirected graphs, by which new graphs are obtained by either removing an edge or shrinking its extremal vertices to a unique vertex (and removing the resulting loop), or vice-versa by adding edges between existing vertices or exploding a vertex into two vertices connected by an edge. Denoting $\dot{\tau}$ the undirected spanning-tree polynomials, the deletion-contraction formula~\cite{bollobas2000contraction} states that $\dot{\tau} = \dot{\tau}^{\del e} + r(e) \dot{\tau}^{\reflectbox{$\setminus$} e}$ and, provided that the spanning-tree polynomial of a graph consisting only of disconnected vertices is $1$, it is constructive (it functions as recursive definition). In fact, it is a specialization to trees of a more general result applying to the Tutte polynomial, that also accounts for other combinatorial properties of graphs beyond trees (e.g.\:one of its evaluations yields the chromatic polynomial)~\cite{tutte2004graph,brylawski1992tutte}. 

Taking $\I = \{+e\}$ in Eq.\,(\ref{eq:gendelcon}) we find that (up to a factor $r(+e)$) the same formula applies to the rooted spanning-tree polynomial
\begin{align}
\tau_\x = \tau^{\del +e}_\x + \tau^{\con +e}_\x,  \label{eq:delcon}
\end{align} 
simply stating the obvious fact that a spanning tree either contains or does not contain a given edge. We propose for this and similar formulas the term deletion-constriction. Notice that this logical acceptation of deletion-constriction is different from the topological acceptation of deletion-contraction for undirected graphs, in that f it does not entail the existence of a graph where some ``shrinking'' produces the desired polynomials. A small immaterial difference between constriction and contraction is that in the latter case the rate of the contracted edge does not appear in the contracted spanning-tree polynomial (thus the difference between Eq.\,(\ref{eq:delcon}) and its undirected analogue).

There is one special case in which constriction becomes contraction: when the root is the source or target of the deleted-contracted edge itself. In this case we have that, provided $\pm e$ both belong to the set of edges,
\begin{align}r(+e) \tau^{\con -e}_{\s(+e)} = r(-e) \tau^{\con +e}_{\s(-e)} = r(-e) r(+e)\dot{\tau}_\e, \label{eq:reps}
\end{align}
where $\dot{\tau}_\e$ is the rooted spanning-tree polynomial of the graph obtained by shrinking edges $+e$ and $-e$ to a unique vertex $\e$, rooted at that vertex (the above equality is intuitive, but see Ref.\,\cite{polettini2017effective} for an explicit algebraic proof).

A consequence of the fact that constriction does not have a graphical representation is that, differently from undirected graphs, to the best of our understanding, deletion-constriction is not recursive and constructive, i.e.\:one cannot build-up the spanning-tree polynomial of a directed graph starting from those of smaller graphs, at least not in the usual intuitive way.

Finally, let us mention that a Tutte-type polynomial for directed graphs has been proposed in Ref.\,\cite{awan2020tutte} satisfying a variant of the deletion-contraction formula. Deletion-contraction is established for signed graphs where each edge carries weight $\pm 1$, and which is related to knots and their invariants~\cite{kauffman1989tutte}.

\subsection{Tree surgery and edge swaps}

\label{sec:swaps}

To obtain our main result we will operate some tree surgery, as introduced in Ref.\,\cite{owen2020universal}. We will only be concerned with edge swaps between two rooted spanning trees $\T_\x$
and $\T_\y$.
The idea is that, if there is no issue with edges that cannot belong to one of the trees while it might belong to the other, by repeatedly swapping edges in a principled manner one obtains two new spanning trees $\T'_\x$
and $\T'_\y$
with the roots interchanged, in such a way that the product of their weights is preserved
\begin{align}
r(\T'_\x) r(\T'_\y) = r(\T_\y) r(\T_\x).
\end{align}

In the process, one of the two trees temporarily becomes a doubly rooted spanning tree $\mathscr{S}_{\x,\y;\z}$, defined as a spanning tree where the orientation of the paths starting from a branching point $\z$ is towards root $\x$ on the one side and towards $\y$ on the other. Notice that a rooted spanning tree can be seen as a degenerate doubly-rooted spanning tree with branching point one of the two roots, e.g.\:$\T_\x = \mathscr{S}_{\x,\y;\y}$. A useful characterization of a rooted spanning tree is as a minimal set of edges connecting all vertices such that each vertex has exactly one outgoing edge (out-degree $1$) apart from one vertex which has none (the root, with out-degree $0$). A non-degenerate doubly-rooted spanning tree is such that one vertex has out-degree $2$ (the branching vertex), two vertices have out-degree $0$ (the roots), and all others have out-degree $1$.


Let us now review the swapping map. We view $\T_\x$ as a degenerate doubly-rooted spanning tree $\S_{\x,\y;\y}$. In $\T_\x$ there is a unique path $\y \path  \x$. Take the first edge $e$ along it. Its removal from $\T_\x$ disconnects $\X$  into two basins of vertices, each spanned by disconnected trees, one with root $\y$ and one with root $\x$. Now identify the same two basins in $\T_\y$. Because $\x$ and $\y$ belong to different basins, along the unique path $\x \path \y$ in $\T_\y$ there is at least one edge that reconnects the two basins\footnote{One may be tempted to intuitively think that there is only one such edge: this is true of path  $\y \path \x$ in $\T_\x$, but not of path $\x \path \y$ in $\T_\y$.}. Take $f$ as the last of these edges, on the temporary assumption that its source vertex $\z$ is different from $\x$. We now swap edges $e$ and $f$ among the two trees. First we map $\T_\y \to \{\T_\y  \setminus f\} \cup e$. Notice that because of the removal of $f$, $\z$ has no outgoing edge. Furthermore, $\y$ now has out-degree $1$ provided by the outgoing edge $e$. The degrees of all other vertices are untouched, and because it is connected, this object is a spanning tree with ``moving'' root $\z$. Second we map $\T_\x \to \{\T_\x \setminus e\} \cup f$. The removal of $e$ deprives $\y$ of its outgoing edge, so now both $\x$ and $\y$ are roots. The addition of $f$ gives $\z$ a second outgoing edge. Thus this object is a doubly-rooted spanning tree $\S_{\x,\y;\z}$ with ``moving'' branching vertex $\z$. Repeat the operation. Now let $e'$ be the first edge on the unique path $\z \path \x$ in $\S_{\x,\y;\z}$. On the second iteration when removing $e'$ and identifying $f'$, the out-degree of $\z$ goes back from $2$ to $1$ while that of the source $\z'$ of $f'$ of goes to $2$, unless $\z'$ is $\x$, in which case its out-degree goes to $1$ and it is no longer a root, and the doubly-rooted spanning tree degenerates into a rooted spanning tree with root $\y$. But this is bound to happen, as the unique path $\x \path \z'$ is strictly contained in the previous path $\x \path \z$, and if we are not yet at $\z' = \x$ we can just repeat the procedure.

Importantly, edge swaps 1) are invertible (see Ref.\,\cite{owen2020universal} for a constructive proof); 2) do not swap an edge whose target is the root of the other tree (in fact, by construction, at any step $e$, $e'$ cannot target $\y$ as they belong to the unique path $\y \leadsto \x$; $f$, $f'$ cannot target $\x$ as $\x$ always belongs to the basin from which they lead out); 3) do not duplicate an edge in an output (doubly rooted) tree. An important consequence of this consideration is that, letting $\A$ and $\B$ be nonintersecting subsets of the edge space not containing the pinned edges $+1,-1$, the above procedure can also be applied to conditioned spanning trees $\T_\x, \T_\y$ furtherly conditioned to pass through $\B$ and not through $\A$ without actually worrying about these other constraints (an example of the consequence of coplanarity in this case was given at the end of Sec.\,\ref{sec:example}, with $\B = {\mathtt{c} \to \mathtt{b}}$). This generalization would be trivial for non-directed graphs, where there always is a graphical representation of contraction. Here it is slightly more subtle.

\subsection{Second-order deletion-constriction formula for a single edge}

Without loss of generality we now ``pin'' one edge $e = 1$ in both directions $+1, -1$, both with nonvanishing weight, assuming it is not a bridge, i.e.\:an edge whose removal disconnects the graph into two subgraphs. We will be concerned with the following conditioned rooted spanning tree polynomials: $\tau_\x^{\del +1, -1}$ rooted in $\x$ not containing either direction of the pinned edge; $\tau_\x^{\del \mp 1 \con \pm 1}$ rooted in $\x$ containing $\pm 1$ but not $\mp 1$; $\tau_{\s(\mp 1)}^{\del +1, -1}$ rooted at the source or target of the pinned edge but not containing the pinned edge; ${ \dot{\tau}}_{\mathtt{1}}$ spanning trees in the graph obtained by contracting the pinned edge to a unique vertex $\mathtt{1}$, rooted at that vertex.

From Eq.\,(\ref{eq:gendelcon}) the following ``first-order'' deletion-constriction formula holds
\begin{align}
\tau_\x = \tau_\x^{\del +1, -1} +  \tau_\x^{\del - 1 \con + 1} +  \tau_\x^{\del + 1 \con - 1},
\label{eq:stst}\end{align}
which intuitively states that a spanning tree either does not pass through neither $+1$ nor $-1$, or passes through either one of them. In fact, a spanning tree cannot pass through both $+1$ and $-1$.

We can now formulate the result illustrated in Fig.\,\ref{fig:treesurgery}.

\begin{figure}
  \centering
  \includegraphics[width=.8\linewidth]{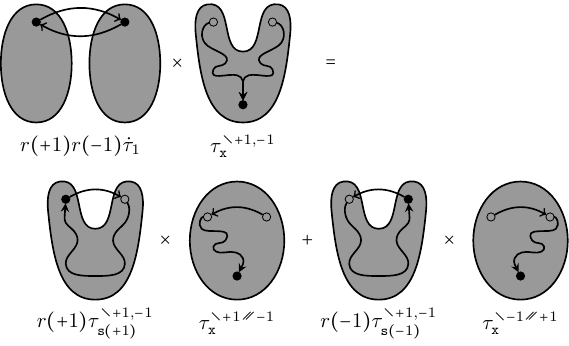}
  \caption{Illustration of the main result, Eq.\,(\ref{eq:delconnew}). Shaded areas represent spanning trees with root in the fully bulleted vertex, while the circles are other vertices different from the root. Directed arrows are specific transitions: if they exit the shaded area their rates are multiplied to the overall tree polynomial, if they are within the shaded area it means that the tree is forced to pass through them. The dashed arrows instead signal that the transitions do not belong to the spanning tree.
  Wiggled arrows denote specific paths within the spanning tree. The reason the lower circled vertex is not represented in some of the terms is because it is undecided whether or not it falls along the wiggled paths (on the lower line), or to which gray area it belongs (on the upper line).}
  \label{fig:treesurgery}
\end{figure}

\begin{theorem}
\label{th:second}
The conditioned spanning tree polynomials satisfy the following ``second-order'' deletion-constriction relation:
\begin{align}
r(+1) r(-1) \dot{\tau}_{\mathtt{1}} \tau_\x^{\del +1, -1} = r(+1) \tau_{\s(+ 1)}^{\del +1, -1} \tau_\x^{\del + 1 \con - 1} + r(-1) \tau_{\s(- 1)}^{\del +1, -1} \tau_\x^{\del -1 \con +1}.
\label{eq:delconnew}
\end{align}
\end{theorem}
\begin{proof}
We proceed by swapping terms between the left-hand side and the right-hand side, but we need to be careful due to edges that we need to pass by, or to avoid.

In the first [resp. second] product on the right-hand side consider a term $r(\T_{\s(\pm 1)}) r(\T_\x)$ with $\T_{\s(\pm 1)} \nowns  (+1, -1)$ and $\T_\x \owns - 1$ [resp $+1$]. We must make sure that edge $\mp 1$ is not swapped. But this cannot happen, as the unique path $\s(\pm 1) \path \x$ cannot contain edge $\mp 1$ as it points towards $\s(\pm 1)$, and the branching vertex is thus allowed to safely navigate to $\x$. Thus we end up with $\T'_{\x} \nowns  (+1, -1)$ and $\T'_{\s(\pm 1)} \owns \mp 1$. Now notice that $r(\T'_{\x} ) r(\T'_{\s(+ 1)}) $ [resp. $r(\T'_{\x} ) r(\T'_{\s(- 1)}) $] is a term in the left-hand side product according to the first [resp. second] representation of $\tau_\mathtt{1}$ in Eq.\,(\ref{eq:reps}).

The other way around, let us now consider a term  of the left-hand side product. We would be tempted to apply the swapping map, but notice that in this case $\pm 1$ might be the final transition in the path $\x \path \s(\mp 1)$, thus producing an unintended edge swap. However, the two representations in Eq.\,(\ref{eq:reps}) grant that we can always re-root from $\s(\mp 1)$ to $\s(\pm 1)$, in which case the re-rooted path $\x \path \s(\pm 1)$ does not contain edge $\pm 1$. Thus we can safely apply the swapping map, and we end up with a term belonging to the right-hand side product. But since the swapping map is invertible, we have established Eq.\,(\ref{eq:delconnew}).
\end{proof}

\subsection{Coplanarity of spanning-tree vectors}
\label{sec:coplanarity}

We now define $\sigma_{\varnothing 1} := - r(+1)r(-1) \dot{\tau}_{\mathtt{1}}$ and $\sigma_{\pm 1} = r(\pm 1) \tau_{\s(\pm 1)}^{\del +1, -1}$ and form the vectors
\begin{align}
\bs{\tau}_\x = \left(\begin{array}{c} \tau_\x^{\del +1, -1} \\ \tau_\x^{\del - 1 \con + 1} \\  \tau_\x^{\del + 1 \con - 1}  \end{array} \right), \qquad
\bs{\sigma}_{1} = \left(\begin{array}{c} \sigma_{\varnothing 1} \\ \sigma_{+1} \\ \sigma_{-1} \end{array} \right), \qquad \bs{1} = \left(\begin{array}{c} 1 \\ 1 \\ 1 \end{array} \right). \qquad  
\label{eq:sigma}
\end{align}
The deletion-constriction formula writes
\begin{align}
\tau_\x  = \bs{1} \cdot \bs{\tau}_\x.
\end{align}

\begin{theorem}
\label{th:coplanarity}
Vectors $\bs{\tau}_\x$, for $\x \in \mathscr{X}$, are coplanar.
\end{theorem}

\begin{proof}
Equation \,(\ref{eq:delconnew}) can be rewritten as
\begin{align}
\bs{\sigma}_{1} \cdot \bs{\tau}_\x = 0 \label{eq:central}
\end{align}
and we conclude.
\end{proof}


\subsection{Mutual linearity among two Markov currents}
\label{sec:mutual}

As an application of co-planarity we now reproduce a recent result obtained by the Authors in Ref.\,\cite{harunari2024mutual} by means of a linear algebra approach developed in Ref.\,\cite{aslyamov2023nonequilibrium}.

Consider a continuous-time Markov chain on a graph $\mathscr{G}$ with time-independent rates $r(\pm e)$ of performing transitions $\pm e \in \E$ between vertices $\x \in \mathscr{X}$ in a short-time interval. We are interested in the mean stationary currents
\begin{align}
\jmath_e 
= r(+e) \, p_{\s({+e})} - r(-e) \, p_{\s({-e})} \label{eq:currents}
\end{align}
where $(p_\x)_{\x \in \mathscr{X}}$ is the unique stationary probability of being at a vertex, whose existence and uniqueness is granted by the assumption of irreducibility. We promote the arbitrarily chosen $\jmath_{1}$ as input current and study the dependence of all other currents $\jmath_e$ on the transition rates $\r_1 = (r(+1), r(-1))$ of the input current, while keeping all other rates fixed. We also assume that the removal of edge $1$ does not disconnect the graph.

In general $\jmath_e(\r_1)$ is a nonlinear function of $\r_1$. We will now show that there exist  parameters $\lambda^0_e$ and $\lambda^1_e$, not dependent on $\r_1$ (but possibly dependent on all other rates), such that all currents are linear-affine in the input current
\begin{align}
\jmath_e(\r_{1}) = \lambda^0_e + \lambda^1_e \, \jmath_1(\r_{1}). 
\label{eq:result}
\end{align}
To do so, we will apply Theorem~\ref{th:coplanarity} on the coplanarity of spanning-tree vectors,  choosing as pinned edge the edge of the input current $\j_1$. Notice also,
since any two currents $\jmath_e$ and $\jmath_{e'}$ are linear-affine in $\jmath_1$, 
Eq.~\eqref{eq:result} implies that%
they are linear-affine among themselves with respect to a variation of the forward and backward rates of an arbitrary edge.

By the Markov-chain matrix-tree theorem~\cite{hill1966studies,schnakenberg1976network,caplan1982tl,king1956schematic} (see also Refs.\,\cite{khodabandehlou2022trees,lecturenotes} for recent approaches) the stationary probability is given by
\begin{align}
p_\x = \frac{\tau_\x}{\sum_\y \tau_\y}.
\end{align}
Plugging into the second expression of Eq.\,(\ref{eq:currents}) one finds that the stationary currents are ratios of a homogeneous polynomial of degree $|\mathscr{X}|$ in the rates over the normalization, which is a homogeneous polynomial of degree $|\mathscr{X}|-1$. An important observation is that both numerator and denominator cannot contain quadratic terms in $r(+1)$ and $r(-1)$, because trees contain no cycles and because at numerator all terms containing products $r(+1) r(-1)$ can be shown to cancel out by application of Eqs.\,(\ref{eq:delcon}) and (\ref{eq:reps}). We can then express the dependency of the currents on $\r_{1}$ as
\begin{align}
\jmath_e(\r_1) = \frac{z^\varnothing_e + z_e^{+}r(+1) + z_e^{-} r(-1)}{z^\varnothing_0 + z_0^{+} r(+1) + z_0^{-} r(-1)}
\label{eq:currentz}
\end{align}
where, reminding the definition of the rescaled spanning-tree polynomial in Eq.\,(\ref{eq:dotted}), we defined
\begin{equation}
\begin{aligned}
z_e^\varnothing & = r(+e) \tau_{\s(+e)}^{\del +1,-1} - r(-e) \tau_{\s(-e)}^{\del +1,-1}  \\
z_e^{\pm} & = r(+e)\dot{\tau}_{\s(+e)}^{\del \mp 1 \con \pm 1 } - r(-e) \dot{\tau}_{\s(-e)}^{\del \mp 1 \con \pm 1 } \\
z_1^\varnothing & = 0 \\
z_1^{\pm} & =  \pm \tau_{\s(\pm 1)}^{\del +1, -1 } & \gtrless 0, \neq 0 \\
z_0^\varnothing & = \sum_\x \tau_{\x}^{\del +1,-1} & > 0 \\
z_0^{\pm} & = \sum_\x \dot{\tau}_{\x}^{\del \mp 1 \con \pm 1} & > 0.
\end{aligned}
\end{equation}
The strict inequalities on the right-hand side are due to the assumption we made that removal of edge $1$ does not disconnect the graph, so that there always exists at least a rooted spanning tree.

Let us now temporarily assume the validity of Eq.\,(\ref{eq:result}). Multiplying by the normalization we find
\begin{multline}
\label{eq:inter}
z^\varnothing_e + z_e^{+} r(+1) + z_e^{-} r(-1) = \lambda^0_e [z^\varnothing_0 + z_0^{+} r(+1) + z_0^{-} r(-1)] +  \lambda^1_e [z_1^{+} r(+1) + z_1^{-} r(-1)].
\end{multline} 
Since this must hold for all $\r_1$ one can disentangle it into the system of linear equations
\begin{align}
\label{eq:zs}
\stackrel{(\bz_0 , \bz_1)}{
\overbrace{\left(\begin{array}{cc}
z_0^{\varnothing} & 0 \\
z_0^{+} & z_1^{+} \\
z_0^{-} & z_1^{-}
\end{array}\right)}}
\stackrel{\bs{\lambda}_e}{
\overbrace{
\left(\begin{array}{c} \lambda^0_e \\ \lambda^1_e \end{array} \right)}
} \; =\; \stackrel{\bz_e}{\overbrace{\left(\begin{array}{c} z_e^{\varnothing} \\ z_e^{+}  \\ z_e^{-}  \end{array} \right)}}.
\end{align}
The system is  overdetermined. By the Rouch\'e--Capelli theorem it affords a unique solution if and only if $(\bz_0 , \bz_1)$ is full rank and the augmented matrix $(\bz_0 , \bz_1, \bz_e)$ has the same rank as $(\bz_0 , \bz_1)$. In this case the rank of $(\bz_0 , \bz_1)$ is $2$, because the determinant  of the upper $2 \times 2$ block is nonvanishing,  $z^\varnothing_0 z_1^+ > 0$. The second condition holds if
\begin{align}
\det (\bz_0 , \bz_1 , \bz_e) = 0, \label{eq:determinant}
\end{align}
on which we come back soon. Vice-versa, if this latter holds, we can find $\bs{\lambda}_e$ satisfying Eq.\,(\ref{eq:result}) by inverting any $2\times 2$ block of the above system. Choosing the upper $2 \times 2$ block we obtain
\begin{align}
\left(\begin{array}{c} \lambda_e^0 \\ \lambda_e^1 \end{array} \right) = 
\left(\begin{array}{cc}
z_0^{\varnothing} & 0 \\
z_0^{+} & z_1^{+} \end{array}\right)^{-1}  \left(\begin{array}{c} z_e^{\varnothing} \\ z_e^{+}
\end{array} \right)   = \frac{1}{z_0^{\varnothing} z_1^{+}}  \left(\begin{array}{c} z_1^+ z_e^\varnothing \\ z^\varnothing_0 z_e^+ - z_0^+ z_e^\varnothing\end{array} \right). 
\end{align}
This expression allows us a consistency check, as for vanishing $\bs{r}_1 = 0$ the input current stalls and we find for the output currents $\j_e(0) = \lambda^0_e = z_e^\varnothing / z_0^\varnothing$, consistently with Eq.\,(\ref{eq:currentz}) and with what previously found e.g.\;in Ref.\,\cite{polettini2017effective}. Furthermore, this also implies that all affine coefficients $\lambda^0_e$ vanish if the graph where $\bs{r}_1 = 0$ satisfies detailed balance.

By Eqs.\,(\ref{eq:zs}), vectors $\bz_e$ and $\bz_0$ are easily seen to be linear combinations of vectors $\bs{\tau}_\x$, therefore by coplanarity Eq.\,(\ref{eq:central}) $\bs{\sigma}_{1}$ is orthogonal to all of them. Furthermore it can be checked directly that $\bs{\sigma}_{1} \cdot \bz_1 = 0$. Therefore matrix $(\bz_0, \bz_1, \bz_e)$ has $\bs{\sigma}_{1}$ as left-null vector and therefore its determinant vanishes, granting our result.

\section{Towards multiple pinned edges}

\subsection{Co-hyper-planarity for two pinned edges}
\label{sec:two}

An obvious question is whether the above results generalize to the deletion-constriction through multiple edges. Let us consider here the case of two pinned edges $\pm 1$ and $\pm 2$ with nonvanishing rates in both directions. 
Here we are not so much interested in special cases and our reasoning will be in part speculative.
However, precise assumptions should be made in order for the argument to not fall at any step. A safe set of conditions is that 1) the graph is large enough (at least $|\mathscr{X}| = 9$ vertices -- the reason being that, as we shall soon see, the spanning-tree vectors have dimension $9$, and we want to have enough of them to span the whole vector space); 2) that the graph remains connected after the removal of pinned edges.

As discussed below, our reasoning also requires a technical assumption that, we conjecture, is not a necessary condition for the result to hold.
%



The deletion-constriction formula in Eq.\,(\ref{eq:gendelcon}) reads $\tau_x = \bs{1} \cdot \bs{\tau}_x$ with $\bs{1}$ a vector with all unit entries and 
\begin{align}
\bs{\tau}_\x = \left(\begin{array}{c}
\tau_\x^{\del +1, -1, +2, -2} \\
\tau_\x^{\del -1, +2, -2 \con +1} \\
\tau_\x^{\del +1, +2, -2 \con -1} \\
\tau_\x^{\del +1, -1, -2 \con +2} \\
\tau_\x^{\del +1, -1, +2 \con -2} \\
\tau_\x^{\del -1, -2 \con +1, +2} \\
\tau_\x^{\del +1, -2 \con -1, +2} \\
\tau_\x^{\del -1, +2 \con +1, -2} \\
\tau_\x^{\del +1, +2 \con -1, -2}
\end{array}\right).
\end{align}

Notice that there are no further nonvanishing vector entries in $\bs{\tau}_\x$ since constriction through both $+e$ and $-e$ yields a vanishing entry, as there do not exist spanning trees that contain an edge in both directions (otherwise they would contain a loop, against their definition).

The question we pose is the dimension of the span of $\{\bs{\tau}_\x\}_{\x \in \mathscr{X}}$. Notice that for $\x = \s(+1)$, $\x = \s(-1)$ and $\x = \s(+2)$ we have
\begin{align}
(\bs{\tau}_{\s(+1)}, \bs{\tau}_{\s(-1)} , \bs{\tau}_{\s(+2)}) = \left(\begin{array}{ccc}
\tau_{\s(+1)} ^{\del +1, -1, +2, -2} & \tau_{\s(-1)}^{\del +1, -1, +2, -2}  & \tau_{\s(+2)}^{\del +1, -1, +2, -2} \\
0 & \tau_{\s(-1)}^{\del -1, +2, -2 \con +1} & \tau_{\s(+2)}^{\del -1, +2, -2 \con +1} \\ 
\tau_{\s(+1)}^{\del +1, +2, -2 \con -1} & 0 & \tau_{\s(+2)}^{\del +1, +2, -2 \con -1} \\
\tau_{\s(+1)}^{\del +1, -1, -2 \con +2} & \tau_{\s(-1)}^{\del +1, -1, -2 \con +2} & 0 \\
\tau_{\s(+1)}^{\del +1, -1, +2 \con -2} & \tau_{\s(-1)}^{\del +1, -1, +2 \con -2} & \tau_{\s(+2)}^{\del +1, -1, +2 \con -2} \\
0 & \tau_{\s(-1)}^{\del -1, -2 \con +1, +2} & 0 \\
\tau_{\s(+1)}^{\del +1, -2 \con -1, +2} & 0 & 0 \\
0 & \tau_{\s(-1)}^{\del -1, +2 \con +1, -2}  & \tau_{\s(+2)}^{\del -1, +2 \con +1, -2}  \\
\tau_{\s(+1)}^{\del +1, +2 \con -1, -2} & 0 & \tau_{\s(+2)}^{\del +1, +2 \con -1, -2}
\end{array}\right). 
\end{align}
The lower $3\times 3$ block has determinant
\begin{align}
\tau_{\s(+1)}^{\del +1, -2 \con -1, +2} \tau_{\s(-1)}^{\del -1, +2 \con +1, -2}  \tau_{\s(+2)}^{\del +1, +2 \con -1, -2} > 0,
\end{align}
therefore at least $3$ of the $\bs{\tau}_\x$ are linearly independent. Notice that the removal of the two pinned edges disconnected the graph, the above determinant would vanish.

We now argue that $3$ is indeed the dimension of the span of $\{\bs{\tau}_\x\}_{\x \in \mathscr{X}}$. 

As briefly mentioned at the end of Sec.\,\ref{sec:swaps}, coplanarity can be generalized to the case the spanning-tree 3-vectors having additional constrictions. First we claim that
there exist vectors $\bs{\sigma}_{1}^\bullet$, $\bs{\sigma}_{2}^\bullet$ of the following form (collected in a matrix) that are orthogonal to all of the $\bs{\tau}_\x$:
\begin{align}
& (\bs{\sigma}_{1}^{\del +2,-2} , \bs{\sigma}_{2}^{\del +1,-1} , \bs{\sigma}_{1}^{\del +2 \con -2} , \bs{\sigma}_{2}^{\del +1 \con -1}, \bs{\sigma}_{1}^{\del -2 \con +2}, \bs{\sigma}_{2}^{\del -1 \con +1} )\\
 & \qquad \qquad = \left(\begin{array}{cccccc}
\sigma_{\varnothing 1}^{\del +2,-2} & \sigma_{\varnothing 2}^{\del +1,-1} & 0 & 0 & 0 & 0 \\
\sigma_{+1}^{\del +2,-2} & 0 & 0 & 0 & 0 & \sigma_{\varnothing 2}^{\del -1 \con +1} \\
\sigma_{-1}^{\del +2,-2} & 0 & 0 & \sigma_{\varnothing 2}^{\del +1 \con -1} & 0 & 0\\
0 & \sigma_{+2}^{\del +1,-1} & 0 & 0 & \sigma_{\varnothing 1}^{\del -2 \con +2} & 0 \\
0 & \sigma_{-2}^{\del +1,-1} & \sigma_{\varnothing 1}^{\del +2 \con -2} & 0 & 0 & 0 \\
0 & 0 & 0 & 0 & \sigma_{+ 1}^{\del -2 \con +2} & \sigma_{+ 2}^{\del -1 \con +1}  \\
0 & 0 & 0 & \sigma_{+ 2}^{\del +1 \con -1} & \sigma_{- 1}^{\del -2 \con +2} & 0 \\
0 & 0 & \sigma_{+ 1}^{\del +2 \con -2} & 0 & 0 & \sigma_{-2}^{\del -1 \con +1}\\
0 & 0 & \sigma_{- 1}^{\del +2 \con -2} & \sigma_{-2}^{\del +1 \con -1} & 0 & 0  
\end{array}\right).
\end{align}
The vector entries can be found as per the definitions at the beginning of Sec.\,(\ref{sec:coplanarity}), but with the additional constraints specified as superscript. A way to impose such constraints starting from the definitions of Sec.~(\ref{sec:coplanarity}) is by taking derivatives with respect to the (logarithm of the) rates of the constricted edges and evaluating at zero with respect to the rates of the deleted edges:
\begin{align}
\sigma_{e}^{\del e_1, e_2 , \ldots \con e'_1 , e_2', \ldots} = \left. \frac{\partial}{\partial \log r(e'_1)} \frac{\partial}{\partial \log r(e'_2)} \ldots \sigma_{e} \right|_{r(e_1) = r(e_2) = \ldots = 0}.
\end{align}
Let us now choose (say) the last one $\bs{\sigma}_{2}^{\del -1 \con +1}$. When multiplied by $\bs{\tau}_\x$ we obtain
\begin{align}
\sigma_{\varnothing 2}^{\del -1, \con +1} \tau_\x^{\del -1 +2, -2 \con +1} + \sigma_{+2}^{\del -1 \con +1} \tau_\x^{\del -1, -2 \con +1, +2} + \sigma_{-2}^{\del -1 \con +1} \tau_\x^{\del -1, +2 \con +1, -2} .
\end{align}
As commented in Sec.\,~\ref{sec:swaps}, the swapping procedure is successful also when dealing with additional constraints. Therefore Eq.\,(\ref{eq:delconnew}) applies, and the above expression vanishes.

It remains to prove that the above $\sigma$-vectors are linearly independent. Rescaling each one of them in such a way that the first entry is $1$ the above matrix now reads
\begin{align}
\left(\begin{array}{cccccc}\label{eq:matrixsigmarescaled}
1 & 1 & 0 & 0 & 0 & 0 \\
A & 0 & 0 & 0 & 0 & 1 \\
B & 0 & 0 & 1 & 0 & 0 \\
0 & C & 0 & 0 & 1 & 0 \\
0 & D & 1 & 0 & 0 & 0 \\
0 & 0 & 0 & 0 & E & F \\
0 & 0 & 0 & G & H & 0 \\
0 & 0 & I & 0 & 0 & J \\
0 & 0 & K & L & 0 & 0
\end{array} \right)
\end{align}
with properly defined $A,\ldots, L$. If this matrix were not full-rank, it would afford a right-null vector $\bs{\omega}$. Applying the upper $5$-block to $\bs{\omega}$ we can constrain its form to
\begin{align}
\bs{\omega} = 
\left( \begin{array}{c}
1 \\
-1 \\
{D} \\
- B \\
{  C} \\
-A 
\end{array}
\right).
\end{align}
Multiplying the vector by the last four rows in Eq.~(\ref{eq:matrixsigmarescaled}) we obtain four consistency requirements for $\bs{\omega}$ to be a null vector. For example, multiplying by the six-th row we find
\begin{align}
0 = EC {  -} FA = \frac{\sigma_{+1}^{\del -2 \con +2} }{\sigma_{\varnothing 1}^{\del -2 \con +2} } \frac{\sigma_{+ 2}^{\del +1,-1} }{\sigma_{\varnothing 2}^{\del +1,-1}} {  -} \frac{\sigma_{+ 2}^{\del -1 \con +1}}{\sigma_{\varnothing 2}^{\del -1 \con +1}} \frac{\sigma_{+ 1}^{\del +2,-2} }{\sigma_{\varnothing 1}^{\del +2,-2}}.
\label{eq:zerocondabsurdum}
\end{align}

We now make the technical (and, we believe, unnecessary) assumption that the graph is fully connected and that all rates are mutually incommensurable, so that there do not happen cancellations just due to a peculiar choice of the rates, unless there is a general law valid for any graphs.
%
If Eq.~\eqref{eq:zerocondabsurdum} and similar ones were true, then
under our technical assumptions each one of them would imply additional second-order deletion-constriction formulas among spanning trees. A rapid randomized numerical check shows that these conditions are violated. Therefore most often the $\sigma$-vectors are all independent, and we can thus confirm that in the case of two edges the dimension of the span of the rooted tree vectors is $3$.
We do not pursue a complete proof of this in the present manuscript.


\subsection{Linearity when controlling two Markov input currents}

This latter hyper-coplanarity for two pinned edges can be employed to obtain the mutual linearity 
of any Markov current when controlling the rates of two input currents:
%
\begin{align}
\j_e(\bs{r}_1, \bs{r}_2) = \mu^0_e + \mu_e^1 \, \j_1(\bs{r}_1, \bs{r}_2) + \mu_e^2 \, \j_2(\bs{r}_1, \bs{r}_2) . \label{eq:holds}
\end{align}
The edges $1$ and $2$ of the input currents are chosen arbitrarily, provided their removal does not disconnect the transition graph.
The procedure is the same as in Sec.\,\ref{sec:mutual}:
We apply the co-hyper-planarity of last section to the pinned edges $1$ and $2$.
A generic current can be expressed as
\begin{align}
\jmath_e(\r_1,\r_2) = \frac{{ \bs{w}}_e \cdot \bs{1}}{\bs{w}_0 \cdot \bs{1}}
\end{align}
where now for sake of elegance we incorporated the explicit dependency on $\bs{r}_1,\bs{r}_2$ in the $\bs{w}$'s defined by 
\begin{align}
{ \bs{w}}_0 & = \sum_\x \bs{\tau}_\x \\
{ \bs{w}}_e & = r(+e) \bs{\tau}_{\s(e)} - r(-e) \bs{\tau}_{\s(-e)}.
\end{align}

Equation (\ref{eq:holds}) holds true if and only if
\begin{align}
( { \bs{w}}_e - \mu^0_e { \bs{w}}_0 - \mu^1_e \bs{w}_1 - \mu^2_e \bs{w}_2 ) \cdot \bs{1} = 0.
\end{align}

Thanks to the result in the previous section we know that, given that its columns are linear combinations of the $\tau$-vectors, matrix $({ \bs{w}}_e , { \bs{w}}_0  , { \bs{w}}_1 , { \bs{w}}_2)$ has rank $3$; matrix $({ \bs{w}}_0  , { \bs{w}}_1 , { \bs{w}}_2)$ is also easily shown to have maximal rank under our speculative assumptions. Then the Rouch\'e--Capelli theorem grants that the system ${ \bs{w}}_e = \mu^0_e { \bs{w}}_0 + \mu^1_e { \bs{w}}_1 + \mu^2_e { \bs{w}}_2$ admits a unique solution, and therefore, running the argument backwards, linearity holds.

\section{Conclusions}

Before discussing possible generalizations and connections to other ideas, let us recap our results. We introduced the concept of deletion-constriction for directed graphs on $n$ pinned edges in both directions, generalizing the well-known concept of deletion-contraction for undirected graphs. This leads to the construction of spanning-tree vectors $\bs{\tau}_\x$ for all roots $\x \in \mathscr{X}$ satisfying the deletion-constriction formula $\tau_\x = \bs{1} \cdot \bs{\tau}_\x$. We suggested the idea that such vectors span a vector space of dimension $n+1$ much smaller than their own dimension, and proved this fact  for $n=1$ in full generality and for $n = 2$, under certain conditions. We used these results to provide an alternative proof of a recently found mutual linearity of two currents for Markov chains on the graph~\cite{harunari2024mutual}, and we discussed its possible generalization to three currents.

Under the simplifying assumption that all edges are reversible, for $n$ pinned reversible edges the spanning-tree vectors have dimension
\begin{align}
\sum_{k = 0}^n 2^n {n \choose k} = 3^n,
\end{align}
where the summation takes into account the fact that no tree can contain both directions of an edge. Given the cases $n = 1$ and $n = 2$, it is tempting to speculate that the dimension of the linear space spanned by them is $n+1$. We collected numerical evidence supporting this conjecture in the case $n = 3$ (in which case the spanning-tree vectors have $27$ entries), for graphs up to $13$ vertices and randomized values of the weights of the unpinned edges (best would be to test it on graphs with $|\mathscr{X}| > 27$ vertices, but this goes beyond our computational capabilities). In this case a direct approach such as the one employed for the $n = 2$ case appears to be prohibitive. Using the product deletion-constriction formula we can produce $n 3^{n-1}$ vectors that are orthogonal to all of the $\bs{\tau}_\x$. The rationale behind this latter number is: there are $n$ $\sigma$-vectors which are unconditioned; there are $2(n-1)$ $\sigma$-vectors conditioned on passing through a given edge, $2(n-2)$ through two edges etc. Hence overall we have
\begin{align}
\sum_{k = 0}^n 2^k {n \choose k} (n-k) = n \sum_{k=0}^n 2^k {n-1 \choose k} = n 3^{n-1}
\end{align}
possible $\sigma$-vectors (in fact for $n=1$ we had $1$ and for $n=2$ we had $6$). But for $n \geq 3$ this is way too many: for example for $n = 3$ the $\tau$-vectors have dimension $27$, but the above procedure would produce $3 \times 3^2 = 27$ $\sigma$-vectors, $4$ of which should then be linearly dependent. We do not have an intuition on how to tame the unindependent ones.

Proving mutual linearity would have been straightforward if instead of varying both forward and backward rates $\bs{r}_1$ we only varied one of them, e.g.\:$r(+1)$. In that case in place of Eq.\,(\ref{eq:inter}) we would have found
\begin{align}
\frac{{z'}^\varnothing_e + z_e^+ r(+1)}{{z'}^\varnothing_0 + z_0^+ r(+1)} = {\lambda'}_e^0 + {\lambda'}_e^1 \frac{{z'}^\varnothing_1 + z_1^+ r(+1)}{{z'}^\varnothing_0 + z_0^+ r(+1)},
\end{align}
and the conclusion follows by noticing that the two ratios are Möbius transformations, and by exploiting properties of the Möbius group. It is particularly fitting that the Möbius group structure arises here as in other areas of physics~\cite{bertoldi2000equivalence} that deal with objects that are intrinsically projective (as are probabilities). One mathematical question, independent of the interpretation of the objects involved as ``currents'' of a Markov process, could then be whether our more general result could be the blueprint to frame some sort of multi-linear Möbius group structure.

\bmhead{Acknowledgments}

We thank J.\,Horowitz for the early involvement in the conjectures regarding tree surgery. Moral support from Guido and Piter from Le Aie Research Institute was also key. The research was supported by the National Research Fund Luxembourg (project CORE ThermoComp C17/MS/11696700), by the European Research Council, project NanoThermo (ERC-2015-CoG Agreement No. 681456), and by the project INTER/FNRS/20/15074473 funded by F.R.S.-FNRS (Belgium) and FNR (Luxembourg). The authors acknowledge support from an CNRS Tremplin project. SDC and VL acknowledge support from IXXI, CNRS MITI and the ANR-18-CE30-0028-01 grant LABS.

\bibliography{biblio}

\end{document}